\documentclass[13pt]{amsart}
\usepackage{epstopdf}% To incorporate .eps illustrations using PDFLaTe\oL, etc.
\usepackage[numbers,sort&compress]{natbib}% Citation support using natbib.sty
\bibpunct[, ]{[}{]}{,}{n}{,}{,}% Citation support using natbib.sty

\makeatletter% @ becomes a letter
\def\NAT@def@citea{\def\@citea{\NAT@separator}}% Suppress spaces between citations using natbib.sty
\makeatother% @ becomes a symbol again
\usepackage{xcolor}%colors

\usepackage{amsmath}
\usepackage{amssymb}
\usepackage{amsthm}
\usepackage{amsfonts}
\usepackage{arydshln}
\usepackage{enumerate}
\usepackage[thinlines]{easybmat}
\numberwithin{equation}{section}
\newcommand{\hsp}{E}

\newcommand{\oL }{\mathcal{L}(\hsp) }

\newcommand{\h}{\mathcal{H}}

\newcommand{\ke}{{\mathbf{k}}}

\newcommand{\M}{{\mathbf{M}}}

\newcommand{\J}{{\mathbf{J}}}

\newcommand{\z}{{z}}
\def\ni{\noindent}
\def\nn{\nonumber}
\newtheorem{theorem}{Theorem}[section]
\newtheorem{proposition}[theorem]{Proposition}
\newtheorem{corollary}[theorem]{Corollary}
\newtheorem{lemma}[theorem]{Lemma}
\theoremstyle{definition}

\newtheorem{example}[theorem]{Example}

\numberwithin{equation}{section}

\title[Toeplitz Operators and Conjugations]{Block Toeplitz operators, Matrix valued truncated Toeplitz Operators, and Conjugations}

\author[R. Khan]{REWAYAT KHAN}
\address{ Department of Artificial Intelligence and Data Engineering, Istanbul Aydin University, Turkey}
\email{rewayatkhan@aydin.edu.tr}

\begin{document}
\begin{abstract}{ In this paper, we characterize the complex symmetry of block Toeplitz operators on the Hardy-Hilbert space $H^2(E)$ with conjugation $C_{\mu, \lambda}$, where $E$ is a finite dimensional Hilbert space. We also extend the application of this conjugation to matrix valued truncated Toeplitz operators. 
	}
\end{abstract}
\maketitle

\section{\bf{Introduction}}
 Let $\mathbb{C}$ be the complex plane and $\mathbb{T}$ the unit circle in $\mathbb{C}$. Let $L^{2}(\mathbb{T})$ be the space of all measurable functions on the unit circle $\mathbb{T}$ whose Fourier coefficients are square summable.
Let $H^2$ be the classical Hardy space on the unit disk $\mathbb{D}=\{z\in\mathbb{C}:|z|<1\}$. The space $H^{2}$ can be thought of as a closed subspace of $L^{2}(\mathbb{T})$, with respect to the normalized Lebesgue measure on $\mathbb{T}$, consisting of functions whose negative Fourier coefficients vanish. %es. 
 The space of essentially bounded functions in $L^2(\mathbb{T})$ is denoted by $L^{\infty}$, while the space of bounded analytic functions in $H^2$ is denoted by $H^{\infty}$.

\medskip

 For a complex separable Hilbert space $E$, let  $\mathcal{L}(E)$ denote the algebra of all bounded linear operators on $E$.  Let $L^{2}(E)$ represent the Hilbert space of $E$-valued measurable functions that are square integrable on $\mathbb{T}$, and let $H^{2}(E)$ denote the corresponding Hilbert Hardy space. For example, when $E=\mathbb{C}^{n}$, we have $L^{2}(E)=L^{2}(\mathbb{C}^{n})$ and $H^{2}(E)=H^{2}(\mathbb{C}^{n})$. Similarly, we can define the $\mathcal{L}(E)$-valued, i.e., operator-valued functions. The corresponding spaces are denoted by $L^{2}(\mathcal{L}(E))$ and $H^{2}(\mathcal{L}(E))$. The operator-valued essentially bounded functions on $\mathbb{T}$ are denoted by $L^{\infty}(\mathcal{L}(E))$, while the space of bounded analytic functions in $H^{2}(\mathcal{L}(E))$ is denoted by $H^{\infty}(\mathcal{L}(E))$. If $\operatorname{dim} E=n<\infty$, then $\mathcal{L}(E)$ corresponds to the space of $n\times n$ matrices. 

\medskip

%A function $B(z)\in \oL$ is called a \textit{finite
%Blaschke-Potapov product} if it admits a factorization of the form
%\begin{equation}\label{BP}
%B=UB_{1}B_{2}...B_{n},
%\end{equation}
%where $U$ is a unitary operator and, for each $j=1,\dots,n$,
%$$B_{j}(z)=\frac{\lambda_{j}-z}{1-%\overline{\lambda}_{j}z}P_{j}+(I-P_{j})$$
%for some $\lambda_{j}\in \mathbb{D}$ and %$P_{j}$ on $E$ is an orthogonal projection. %s , and a unitary operator $U$.
%The degree of the Blaschke-Potapov product in %\eqref{BP} is defined by
%$$\text{deg\,} B
%=\sum\limits_{j=1}^{n}\text{rank\,} P_{j}.$$
% Note that $\Theta$ in \eqref{BP} is an inner %function.

 A vector valued model space $K_{\Theta}\subset H^{2}(\hsp)$ is the orthogonal complement of $\Theta H^{2}(\hsp)$, that is, $K_{\Theta}= H^{2}(\hsp)\ominus  \Theta H^{2}(\hsp)$. Here $\Theta$ is an operator-valued inner function: a function with values in $\oL $ (the algebra of all bounded linear operators on $\hsp$), analytic in $\mathbb{D}$, bounded, and such that the boundary values $\Theta(\z)$ are unitary operators a.e. on $\mathbb{T}$. These spaces appear in connection with the model theory of Hilbert space contractions (see \cite{NF}). 
\medskip

The paper is organized as follows: after providing some introductory materials in Section 1 and 2, we characterize the complex symmetry of block Toeplitz operators with
respect to the given conjugation in Section \ref{s3}. In Section \ref{s4}, we extend the application of the given conjugation to matrix valued truncated Toeplitz operators.

\section{Preliminaries}\label{s2}
Let $E$ be a complex separable Hilbert space. A map $C:E\to E$ is called a conjugation if 
\begin{enumerate}
    \item $C$ is antilinear, that is, $C(\alpha f+ \beta g)=\bar \alpha Cf+ \bar \beta Cg$   for all $\alpha, \beta \in \mathbb{C}$ and $f,g\in E$,
    \item $C$ is an involution, that is, $C^2=I_{E}$, and 
    \item $C$ is isometric, that is $\|Cf\|=\|f\|$ for all $f\in E$. 
\end{enumerate}
The importance of
conjugations comes, for example, from their connection with complex
symmetric operators. Recall that a bounded linear operator $T\in \oL$ is said to be $C$-symmetric ($C$ being a conjugation on $\hsp$) if $CTC=T^*$. We say that $T$ is complex symmetric if it is $C$-symmetric with respect
to some conjugation $C$ (see, e.g., \cite{GP} for more details on conjugations
and complex symmetric operators).

We denote the norm and inner product in $E$ by $\|.\|_{E}$ and $\langle . , .\rangle_{E}$, respectively. The space $L^2(E)$ consists of functions $f:\mathbb{T}\to E$ of the form
\begin{equation}\label{one1}
 f(z)=\sum\limits_{n=-\infty}^{\infty}a_{n}z^n
 \text{ such that}  \sum\limits_{n=-\infty}^{\infty}\|a_{n}\|_{E}^{2}<\infty\end{equation}
  $\text{with}\{a_{n}\}\subset E$.
It is a separable Hilbert space with the inner product given by
\begin{equation*}
\langle f,g\rangle_{L^{2}(E)}=\int_{\mathbb{T}}\langle f(z),g(z)\rangle_{E}\,dm(z),\quad f,g\in L^2(E).
\end{equation*}
If $f\in L^2(E)$ is as in %given by
 \eqref{one1}, then its Fourier series converges in the $L^2(E)$-norm and
$$\|f\|_{L^2(E)}^2=\int_{\mathbb{T}} \| f(z)\|_{E}^2\,dm(z)=\sum\limits_{n=-\infty}^{\infty}\|a_{n}\|_{E}^{2}.$$

The vector-valued Hardy-Hilbert space %$H^{2}(E)$ 
is a subspace 
 of $L^2(E)$ defined as
$$H^{2}(E)={\{f\in L^{2}(E): a_{n}=0, n\leq -1}\}.$$
\smallskip

Throughout the paper, $E$ is a finite dimensional Hilbert space with $\text{dim}~E=n$. Let $\{e_j\}_{j=1}^{n}$ be a fixed orthonormal basis for $E$. Then the set $$\mathcal{B}=\{z^ke_{j}: k\geq 0~\text{and} ~j=1,2,...,n\}.$$ is orthonormal basis for $H^2(E)$.

%For each $w \in \mathbb{D}$ and $x\in E$, the function $$\ke_{w} x=\frac{1}{1-\bar w z}I_{E}\,x,$$ belong to $H^2(E)$ is a reproducing kernel function for $H^2(E)$ and has the following reproducing property
%$$\langle f, \ke_{w}x\rangle=\langle f(w), x\rangle.$$

%Similarly, 
The space $L^{2}(\oL$ is defined as %, i.e.,
$$L^{2}(\oL={\Big\{\Phi:\mathbb{T}\to \oL: \Phi(z)=\sum\limits_{n=-\infty}^{\infty}\Phi_{n}z^n, 
\sum\limits_{n=-\infty}^{\infty}\|\Phi_{n}\|_{\oL}^{2}<\infty }\Big\},$$
while %We can define
 the space $H^{2}(\oL$ is the set of $\Phi\in L^{2}(\oL$ represented as above, with $\Phi_{n}=0$ for all $n\leq -1$. 
%$$H^{2}(\LL(E))={\{\Phi\in L^{2}(\LL(E)): A_{n}=0, n\leq -1}\}.$$

The space of essentially bounded functions in $L^{2}(\oL$ is denoted by $L^{\infty}(\oL$, while the space of bounded functions in $H^{2}(\oL$ is denoted by $H^{\infty}(\oL$.

Note that for each $w\in \mathbb{D}$, the function $\ke_w(z)= (1-\bar wz)^{-1}I_{\hsp}$ belongs to $H^{\infty}(\oL )$. Moreover, for every $x\in\hsp$, the function $\ke_wx:z\mapsto \ke_w(z)x$ belongs to $H^{\infty}(\hsp)$ and has the following reproducing property:
$$\langle f, \ke_w x\rangle_{L^2(\hsp)}=\langle f(w),x\rangle_{\hsp}, \quad f\in H^2(\hsp).$$
For each $w\in\mathbb{D}$ we can consider
$$\ke_w^{\Theta}(z)=\tfrac1{1-\bar w z}(I_\hsp-\Theta(z)\Theta(w)^*)\in H^{\infty}(\oL).$$
For $x\in\hsp$ we will denote the function $z\mapsto \ke_w^{\Theta}(z)x$ simply by $\ke_w^{\Theta}x$. Then, for each $x\in\hsp$ and $w\in\mathbb{D}$, the function $\ke_w^{\Theta}x=P_{\Theta}(\ke_wx)$  belongs to $K_{\Theta}^{\infty}=K_{\Theta}\cap H^{\infty}(\hsp )$ and has the following reproducing property:
$$\langle f, \ke_w^{\Theta}x\rangle_{L^2(\hsp)}=\langle f(w), x\rangle_\hsp\quad\text{for every }f\in K_{\Theta}.$$
It follows in particular that $K_{\Theta}^{\infty}$ is a dense subset of $K_{\Theta}$.

%Moreover, we also have 
%$$\widetilde{k_{w}^{\Theta}}(z)x=\frac{\Theta(z)-\Theta(w)}{z-w}x.$$

Let $\Phi\in L^{\infty}(\mathcal{L}(E))$. The multiplication operator $M_\Phi: L^2(E)\to L^2(E)$ is defined by
$$M_{\Phi}f=\Phi f, \quad ~~f\in L^{2}(E).$$
The  block Toeplitz operator $T_{\Phi}$ is defined by
$$T_{\Phi}f=P_{+}(M_\Phi f),\quad f\in H^{2}(E),$$
where $P_{+}$ denotes the orthogonal projection from $L^{2}(E)$ onto $H^{2}(E)$. The operator $S=T_{zI_{E}}$ is an example of a block Toeplitz operator.

Let $P_{\Theta}$ be the orthogonal projection from $L^2(E)$ onto $K_{\Theta}$.
For an operator valued inner functions $\Theta\in H^{\infty}(\oL)$ and $\Phi\in L^2(\oL )$, let
\begin{equation}\label{11}
A_{\Phi}^{\Theta} f=P_{\Theta} (\Phi f),\quad f\in K_{\Theta}\cap H^{\infty}(E).
\end{equation}
 The operator $A_{\Phi}^{\Theta}$ is called a matrix valued truncated Toeplitz operator (see \cite{KT}). It is densely defined on $K_{\Theta}$. Let $\mathcal{MT}(\Theta)$ be the set of all MTTOs of the form \eqref{11} that can be extended boundedly to the whole space $K_{\Theta}$.

\section{Complex symmetry of block Toeplitz operators}\label{s3}
A bounded linear operator $T$ on a separable Hilbert space $E$ is complex symmetric if there exists an orthonormal basis for $E$ with respect to which $T$ has a self-transpose matrix representation. Complex symmetric operators on Hilbert spaces are the generalization of complex symmetric matrices. The study of complex symmetric operators was initiated by Garcia, Putinar, and Wogen (\cite{GP, GPII, GW}). The class of complex symmetric operators includes all normal operators, Hankel operators, truncated Toeplitz operators, Volterra integration operators, and many others. We refer the reader to \cite{GP, GPII, GW} for more details.

%An equivalent definition also exist and involve conjugation.
%A \emph{conjugation} on a Hilbert space $E$ is a conjugate-linear, isometric and involutive map. We say that $T$ is $C$-symmetric if $T=CT^{*}C$,
%and complex symmetric if there exists a conjugation $C$ with respect to which $T$ is $C$-symmetric (see~\cite{GP}).

Let $\mu, \lambda\in \mathbb{T}$. The antilinear operator  $C:H^{2}\to H^{2}$ defined by
\begin{eqnarray}\label{CON}
(Cf)(z)=\mu \overline{f(\lambda \bar z)},
\end{eqnarray}
is a conjugation on $H^{2}$ \cite[Lemma 2.3]{KL}.

\medskip

The following theorem provides a useful criterion to determine whether a standard Toeplitz operator is complex symmetric  with respect to the conjugation defined in \eqref{CON} (see \cite[Theorem 2.2]{KKL}).

\begin{theorem}\cite[Theorem 2.2]{KKL}\label{STO}
Let $\varphi\in L^{\infty}$. Then $T_{\varphi}$ is complex symmetric with respect to the conjugation $C$ if and only if $$\varphi(z)=\varphi(\lambda \bar z).$$
\end{theorem}

In this section, we do not fix any specific conjugation on $E$, but rather consider a more general conjugation on $E$, which leads to complex symmetry of block Toeplitz operators.

\medskip

\begin{lemma}
Let $\mu, \lambda\in \mathbb{T}$ and $\Gamma$ be a conjugation on $E$. Then the operator $C_{\mu, \lambda}:L^2(E)\to L^2(E)$ defined by 
\begin{equation}\label{Gam}
(C_{\mu, \lambda}f)(z)=\mu \Gamma (f(\lambda \bar z)), \quad f\in L^2(E)
\end{equation}
is a conjugation. 
\end{lemma}

\begin{proof} Let $\Gamma$ be a conjugation on $E$. Indeed, the operator $C_{\mu,\lambda}$ is antilinear as $\Gamma$ is a conjugation.
The operator $C_{\mu, \lambda}$ is an involution because 
$$C_{\mu, \lambda}^2 f(z)=C_{\mu, \lambda}(\mu \Gamma(f(\lambda \bar z)))=\bar \mu \mu \Gamma^2(f(\lambda \bar \lambda z))=f(z).$$
 Next, we show that $$\langle C_{\mu,\lambda}f, C_{\mu,\lambda}g\rangle=\langle g, f\rangle, ~~~\text{for all} ~~f,g\in L^{2}(E).$$ Consider
\begin{align*}
\langle C_{\mu,\lambda}f, C_{\mu,\lambda}g\rangle
&=\langle\mu \Gamma(f(\lambda\bar z)) , \mu \Gamma(g(\lambda\bar z))\rangle
=\langle \Gamma(f(\lambda\bar z)) ,  \Gamma(g(\lambda\bar z))\rangle\\
&=\langle  g(\lambda\bar z), f(\lambda\bar z) \rangle
=\langle  g(z), f(z) \rangle\\
&=\langle  g, f \rangle.
\end{align*}
Hence, $C_{\mu,\lambda}$ is a conjugation %on $L^{2}(E)$.
\end{proof}

Let $\{e_j\}_{j=1}^{n}\subset \hsp$ be an orthonormal basis for $\hsp$ and $\Gamma$ be a conjugation on $E$. Then, by \cite[Proposition 2]{SJ}, there exists a unitary operator $U\in \oL$ and a conjugation $J$ on $\hsp$ defined by 
\begin{equation}\label{j}
    J\Big(\sum\limits_{j=1}^{n}\alpha_j e_j\Big)=\sum\limits_{j=1}^{n}\overline{\alpha}_j e_j
\end{equation}
such that 
\begin{equation}\label{UJ}
\Gamma=UJ.
\end{equation}

Let $\h=H^2(E)$ and $J$ be a conjugation \eqref{j} on $E$. 
 The conjugation
  $\J^*:H^2(E)\to H^2(E)$ is defined by
\begin{equation}\label{J}
(\J^*f)(z)=J(f(\overline{z}))\quad \text{a.e. on}~~\mathbb{T}.
\end{equation}
It is not difficult to verify that for $f(z)=\sum_{n=0}^{\infty}a_nz^n\in H^{2}(E)$ we have
\begin{equation}
(\J^*f)(z)=\sum_{n=0}^{\infty}J(a_n)z^n.
\end{equation}

Hence, $\J^*$ is an $M_z$-commuting conjugation, i.e., $\J^*M_z=M_z\J^*$, $\J^*(H^{2}(\hsp))=H^{2}(\hsp)$, and $\J^*P_+=P_+\J^*$.
\begin{proposition}\label{adj}
Let $\mu, \lambda\in \mathbb{T}$, and $U\in \mathcal{L}(E)$ be a unitary operator. Then the operator $V_{\mu, \lambda}:H^2(E)\to H^2(E)$ defined by 
$$(V_{\mu, \lambda}f)(z)=\mu (U f)(\overline{\lambda} z)$$ is unitary.  
\end{proposition}

\begin{proof}
The operator $V_{\mu, \lambda}$, being a composition of three unitary operators, is unitary. 
\end{proof}

\begin{lemma}
The conjugation $C_{\mu, \lambda}$ can be factorized as  $C_{\mu, \lambda}=V_{\mu, \lambda}\J^*=\J^*V_{\mu, \lambda}^*$, where $\J^*$ is given by \eqref{J}.
\end{lemma}

\begin{proof}

Let $\Gamma$ be a conjugation on $E$. Then, by \eqref{UJ}, we have 
$$\Gamma=UJ.$$ Moreover, we have
\begin{align*}
    (V_{\mu, \lambda}\J^{*}f)(z)
    &=V_{\mu, \lambda} \left(\sum\limits_{n=0}^{\infty}J(a_n)z^n\right)\\
    &=\mu \sum\limits_{n=0}^{\infty}UJ(a_n)\overline{\lambda}^nz^n\\
    &=\mu \sum\limits_{n=0}^{\infty}\Gamma(a_n)\overline{\lambda}^nz^n\\
    &=\mu \Gamma (f(\lambda \bar z))\\
    &=(C_{\mu, \lambda}f)(z).
\end{align*}
Hence, we have $C_{\mu, \lambda}=V_{\mu, \lambda}\J^*$. Similarly, we have $C_{\mu, \lambda}=\J^*V_{\mu, \lambda}^*=\J^*V_{\overline{\mu}, \overline{\lambda}}$. 
\end{proof}

The action of $C_{\mu, \lambda}$ on the basis vectors is given by 
\begin{equation}\label{b1}
    C_{\mu, \lambda}(z^k e_{j})=\mu \Gamma (\lambda^k \overline{z}^ke_j)=\mu \overline{\lambda}^k z^k \Gamma e_j=\mu \overline{\lambda}^k z^k UJ e_j.
\end{equation}
Next we compute 
\begin{equation}\label{b2}
    V_{\mu, \lambda}(z^k e_j)=\mu U(\overline{\lambda}^k z^k e_j)=\mu \overline{\lambda}^k z^k Ue_j,
\end{equation}
 for all $j=1,2,3,\dots,n$ and $k\geq 0$.
It follows that 
$$C_{\mu, \lambda}(z^k e_{j})= V_{\mu, \lambda}(z^k e_j)=\mu \overline{\lambda}^k z^k UJ e_j.$$
If $Je_j=e_j$ or equivalently 
$$\Gamma e_j=Ue_j=f_j.$$ 

Since $V_{\mu, \lambda}$ is unitary, the set 
$$\mathcal{B}^{'}=\{\mu \overline{\lambda}^k z^k Ue_j:j=1,2,3,...,n, k\geq 0\}.$$
is an orthonormal basis for $H^2(E)$.

Let $M_{z}$ be the shift corresponding to the basis 
$\mathcal{B}$. Then we have 
$$S=C_{\mu, \lambda}M_z C_{\mu, \lambda},$$
which is a shift corresponding to the basis $\mathcal{B}^{'}$.

Also we have 
$$\J^*\M_z=\M_z\J^*~\text{and}~C_{\mu,\lambda}=V_{\mu, \lambda}\J^*=\J^* V^{*}_{\mu, \lambda}.$$ Therefore, we have 
$$S=C_{\mu, \lambda}M_z C_{\mu, \lambda}=V_{\mu, \lambda}\J^* M_z \J^* V^{*}_{\mu, \lambda}=V_{\mu, \lambda} M_z V^{*}_{\mu, \lambda}.$$

%Hence,
%\begin{align*}
%Sf_{k,j}
%&=S(\mu \overline{\lambda}^k z^k Ue_j)\\
%&=V_{\mu, \lambda} M_z V^{*}_{\mu, \lambda}(\mu \overline{\lambda}^k z^k Ue_j)\\
%&=V_{\mu, \lambda} M_z (z^k e_j)\\
%&=V_{\mu, \lambda} (z^{k+1} e_j)\\
%&=\mu U ( (\overline{\lambda}z)^{k+1} e_j)\\
%&=\mu \overline{\lambda}^{k+1}z^{k+1} %Ue_j\\
%&=f_{k+1,j}
%\end{align*}

 Let $\Gamma$ be an antilinear operator on $E$. Then $\Gamma$ is a conjugation on $E$ if and only if there exists an orthonormal basis $\{f_j\}_{j=1}^{n}$ for $E$ such that $$\Gamma \Big(\sum_{i=1}^{n}\alpha_i g_i\Big)=\sum_{i=1}^{n}\sum_{j=1}^{n}\bar{\alpha}_{i}c_{i,j}^{(g)}g_{j},$$ 
for every orthonormal bases $g=\{g_j\}_{j=1}^{n}$ of $E$, $\alpha_i\in \mathbb{C}$, where 
$$c_{i,j}^{(g)}=\sum_{k=1}^{n}\langle f_k, g_i \rangle\langle f_{k}, g_j \rangle.$$

 Note that, for any conjugation $\Gamma$ on $E$, there exists an orthonormal basis $\{e_{j}\}_{j=1}^{n}$ such that $\Gamma e_j=e_j$. Throughout this paper, we consider the conjugation $\Gamma$ on $E$ satisfying $\Gamma e_j=e_j$. 
\begin{example}
Let $\{e_j\}_{j=1}^{n}$ be an orthonormal basis for $E$ such that $$\Gamma e_j=e_j.$$ Consider the basis $$f_{k,j}=\mu^{\frac{1}{2}}\overline{\lambda}^{\frac{k}{2}}z^k e_j.$$ Let $C_{\mu, \lambda}$ be the conjugation on $H^2(E)$. Then $$C_{\mu,\lambda}f_{k,j}=\mu^{\frac{1}{2}}\overline{\lambda}^{\frac{k}{2}}z^k \Gamma e_j=\mu^{\frac{1}{2}}\overline{\lambda}^{\frac{k}{2}}z^k e_j=f_{k,j}.$$
By \cite[Section 3]{SJ}, we can compute 
$c^{(\mathcal{B})}_{(k,j), (m,l)}$ in the basis 
$$\mathcal{B}=\{z^ke_{j}: k\geq 0~\text{and} ~j=1,2,...,n\}$$
$$c^{(\mathcal{B})}_{(k,j), (m,l)}=\sum_{r,p}\langle f_{r,p}, z^k e_j \rangle \langle f_{r,p}, z^m e_l \rangle.$$
We obtain 
$$\langle f_{r,p}, z^k e_j\rangle =\mu^{\frac{1}{2}}\overline{\lambda}^{\frac{r}{2}}\delta_{r,k}\delta_{p,j}.$$
Therefore, the only nonzero term is 
$$c^{(\mathcal{B})}_{(k,j), (m,l)}=
    \begin{cases}
         \mu \overline{\lambda}^{k}, &(k, j)=(m,l),\\
         0, &\quad \text{otherwise}.
    \end{cases}
   $$
   Consequently, for 
   $$f(z)=\sum_{k= 0}^{\infty}\sum_{j=1}^{n}a_{k,j}z^ke_j,$$  and $a_{k}\in E$ we have 
   
   $$\Gamma(a_{k})=\sum_{j=1}^{n}\bar a_{k,j}e_j.$$ Therefore, we get the conjugation
   \begin{align*}
   C_{\mu, \lambda}f(z)
&=\sum_{k,j}\sum_{m,l}\overline{a}_{k,j}c^{(B)}_{(k,j), (m,l)}z^m e_l\\
&=\sum_{k= 0}^{\infty}\sum_{j=1}^{n}\overline{a}_{k,j}\mu \overline{\lambda}^k z^k e_{j}\\
&=\mu\sum_{k=0}^{\infty}\Gamma (a_{k})\overline{\lambda}^k z^k\\
&=\mu\Gamma (f(\lambda \bar z)). 
   \end{align*}
\end{example}

The following result shows that the conjugation $C_{\mu,\lambda}$ commutes with the orthogonal projection $P_{+}$.

\begin{proposition}
Let $P_{+}$ denote the orthogonal projection from $L^{2}(E)$ onto $H^{2}(E)$. Then the conjugation $C_{\mu,\lambda}$ and $P_{+}$ commute.
\end{proposition}
\begin{proof} Let $f\in L^2(E)$. Then
$$f(z)=\sum_{-\infty}^{\infty}a_{n}z^n,$$ and 
$$(C_{\mu, \lambda}f)(z)=\mu \sum_{-\infty}^{\infty}\bar \lambda^{n} \Gamma(a_{n})z^n.$$
Therefore,
$$P_{+}C_{\mu, \lambda}f=\mu \sum_{0}^{\infty}\bar \lambda^{n} \Gamma(a_{n})z^n=C_{\mu, \lambda}P_{+}f,$$
which proves the required result. 
\end{proof}

Theorem \ref{CSTO} is one of the main result of this section. It is a generalization of \cite[Theorem 2.2] {KKL}. In other words, it characterizes the complex symmetry of block Toeplitz operators with respect to the conjugation $C_{\mu, \lambda}$.

%Let $\Gamma$ be a conjugation on $E$, and let $\Phi\in L^{\infty}(\mathcal{L}(E))$. We say that $\Phi(e^{it})$ is $\Gamma$-symmetric if $\Gamma \Phi(e^{it}) \Gamma=\Phi(e^{it})^{*}$,~a.e. on $\mathbb{T}$.

\begin{theorem}\label{CSTO}
Let $\Phi\in L^{\infty}(\mathcal{L}(E))$ and $\Gamma$ be a conjugation on $E$. Let $C_{\mu,\lambda}$ be a conjugation as  defined in \eqref{Gam}. Then
 $T_{\Phi}$ is complex symmetric with respect to the conjugation $C_{\mu,\lambda}$ if and only if $\Gamma \Phi(\lambda \bar z)\Gamma=\Phi(z)^{*}$.
  \end{theorem}
\begin{proof}
Let $\Phi\in L^{\infty}(\mathcal{L}(E))$. Then for any $f\in H^{2}(E)$, we have
\begin{align*}
C_{\mu,\lambda}T_{\Phi}C_{\mu, \lambda}f
&=C_{\mu, \lambda}T_{\Phi}(\mu \Gamma f(\lambda \bar z))\\
&=C_{\mu,\lambda}P_{+}(\mu \Phi \Gamma f(\lambda \bar z))\\
&=P_{+}\bar \mu C_{\mu, \lambda}( \Phi \Gamma f(\lambda \bar z))\\
&=P_{+}( \Gamma \Phi(\lambda \bar z) \Gamma f(z))\\
&=T_{\Gamma \Phi(\lambda \bar z)\Gamma}f.
\end{align*}
%It follows that $C_{\mu,\lambda}\Phi f -\Phi^{*} C_{\mu,\lambda}f\in H^{2}_{-}(E)=[H^{2}(E)]^{\perp}$. Therefore, for every $h\in H^{2}(E)$, we have
%$$\langle (C_{\mu,\lambda}\Phi f   -\Phi^{*} C_{\mu,\lambda}f), h\rangle=0,$$
%i.e.,
%$$\langle \mu\Gamma \Phi(\lambda \bar z)f(\lambda \bar z)   -\Phi^{*}(z) \mu\Gamma f(\lambda \bar z) , h\rangle=0,$$
%and
%\begin{equation}\label{con}
%\langle \Gamma \Phi(\lambda \bar z)f(\lambda \bar z)   -\Phi^{*}(z)\Gamma f(\lambda \bar z) , h\rangle=0.
%\end{equation}
%Equation \eqref{con} holds only when $\Gamma \Phi(\lambda\bar z)\Gamma=\Phi(z)^{*}$ a.e. on $\mathbb{T}$.
 %   }
It follows that $T_{\Phi}$ is complex symmetric with respect to the conjugation $C_{\mu, \lambda}$ if and only if $$T_{\Gamma\Phi(\lambda \bar z)\Gamma}=T_{\Phi(z)^{*}}.$$
Hence, $T_{\Phi}$ is complex symmetric if and only if $$\Gamma \Phi(\lambda\bar z)\Gamma=\Phi(z)^{*}\quad \text{a.e. on }~\mathbb{T}.$$
\end{proof}

For $\Phi\in L^{\infty}(\oL)$, we can write 
\begin{align*}
\Phi(z)
&=\sum_{-\infty}^{\infty}\Phi_{n}z^n\\
&=\sum_{n=-\infty}^{-1}\Phi_{n}z^n+\Phi_{0}+\sum_{n=1}^{\infty}\Phi_{n}z^n\\
&=\sum_{n=0}^{\infty}\Phi_{-n}\bar z^n+\sum_{n=1}^{\infty}\Phi_{n}z^n\\
&=\Phi_{+}(z)+\Phi_{-}(z)^{*}
%&=\sum_{n=1}^{\infty}  \Phi_{n}\lambda^n\bar %z^n+\sum_{n=0}^{\infty}\Phi_{n}z^n.
\end{align*} 
where 
$$\Phi_{+}=\sum_{n=1}^{\infty}\Phi_{n}z^n ~~\text{and} ~~\Phi_{-}(z)=\sum_{n=0}^{\infty}\Phi_{-n}^{*}z^n.$$
The adjoint of $\Phi$ is given by 
\begin{align*}
\Phi(z)^{*}
&=\sum_{n=0}^{\infty}  \Phi_{-n}^* z^n+\sum_{n=1}^{\infty}\Phi^{*}_{n}\bar z^n\\
&=\sum_{n=1}^{\infty}  \Phi_{n}^*\bar z^n+\sum_{n=0}^{\infty}\Phi^{*}_{-n}z^n
\end{align*}

\medskip

 In Theorem \ref{Gen1}, we characterize the complex symmetry of block Toeplitz operators $T_{\Phi}$ with respect to the $\Gamma$ twisted conjugation $C_{\mu, \lambda}$ on $H^2(E)$, where $E$ is finite dimensional Hilbert space.

\begin{theorem}\label{Gen1}
Let $\Phi(z)=\sum_{n=-\infty}^{\infty}\Phi_{n}z^n\in L^{\infty}(\oL)$. Then the following statements are equivalent:
\medskip

{\em (i)} $T_{\Phi}$ is complex symmetric with respect to the conjugation $C_{\mu,\lambda}$.\par
\medskip

{\em (ii)} $\Phi_{-n}=\lambda^n \Gamma\Phi_n^{*}\Gamma$ for $n\ge0$.\par
\medskip
{\em (iii)} $\Phi(z)=\sum_{n=1}^{\infty}\Phi_n z^n+\sum_{n=0}^{\infty}\lambda^n\Gamma \Phi_{n}^{*}\Gamma  \overline{z}^n.$\par
\medskip
{\em (iv)} $\Phi(z)=\Phi_{+}(z)+\Phi_{0}+\Gamma \Phi_{+}(\lambda \bar z)^{*}\Gamma $, where $\Phi_{+}\in z H^2(\oL)$.
\end{theorem}
\begin{proof}
$(i)\Leftrightarrow (ii)$ 
%Assume that $T_{\Phi}$ is complex symmetric with respect to the conjugation $C_{\mu, \lambda}$. Then $$C_{\mu, \lambda} T_{\Phi}=T^{*}_{\Phi}C_{\mu, \lambda}.$$
We know that the set 
$$B=\{z^k e_j: k\geq 0, ~j=1,2,...,n\}$$ is an orthonormal basis for $H^2(E)$. Then
\begin{align}
C_{\mu, \lambda}T_{\Phi}(z^k e_j)
&=\mu \sum_{n=-k}^{\infty}\overline{\lambda}^{(n+k)}z^{n+k}\Gamma \Phi_{n}e_j\nn\\
&=\mu \sum_{n=-k}^{\infty}\overline{\lambda}^{(n+k)}z^{n+k}\Gamma \Phi_{n}\Gamma e_j.\label{CT}
\end{align}
For $$\Phi(z)=\sum_{-\infty}^{\infty}\Phi_{n}z^n, \quad \text{we have}\quad 
\Phi(z)^{*}=\sum_{-\infty}^{\infty}\Phi_{-n}^{*}z^n.$$ Also, we have 
\begin{align}
T^{*}_{\Phi}C_{\mu, \lambda}(z^k e_j)
&=\mu \overline{\lambda}^kT_{\Phi}^{*}(z^k e_j)\nn\\
&=\mu \overline{\lambda}^kT_{\Phi^*}(z^k e_j)\nn\\
&=\mu \sum _{n=-k}^{\infty}\overline{\lambda}^k z^{n+k}\Phi^{*}_{-n} e_j.\label{TC}
\end{align}
By \eqref{CT} and \eqref{TC}, $C_{\mu, \lambda} T_{\Phi}=T^{*}_{\Phi}C_{\mu, \lambda}$ if and only if
$$\Phi_{-n}=\lambda^n \Gamma \Phi_{n}^{*}\Gamma \quad \text{for} ~~n\geq 0.$$ 

\ni $(i)\Leftrightarrow (iii)$ Let $T_{\Phi}$ is complex symmetric. Then by $(ii)$ we have 
\[\Phi_{-n}=\lambda^n \Gamma \Phi_{n}^{*}\Gamma \quad \text{for}~~n\geq 0.\]
Now
\begin{align*}
\Phi(z)
&=\sum_{-\infty}^{\infty}\Phi_{n}z^n\\
&=\sum_{n=-\infty}^{0}\Phi_{n}z^n+\sum_{n=1}^{\infty}\Phi_{n}z^n\\
&=\sum_{n=0}^{\infty}\Phi_{-n}\bar z^{n}+\sum_{n=1}^{\infty}\Phi_{n}z^n\\
&=\sum_{n=1}^{\infty}\Phi_{n}z^n+\sum_{n=0}^{\infty}\lambda^n\Gamma \Phi_{n}^{*}\Gamma \bar z^n.
\end{align*}
Conversely,  we recall that for $$\Phi(z)=\sum_{n=-\infty}^{\infty}\Phi_{n}z^n$$ where 
\begin{equation}\label{coeff}
\Phi_{n}x=\frac{1}{2\pi}\int_{0}^{2\pi}\bar z^n \Phi(z)x~dt\quad (n=0,\pm1, \pm 2,...)
\end{equation}
for $x\in E$.
Let $$\Phi(z)=\sum_{m=1}^{\infty}\Phi_{m}z^m+\sum_{m=0}^{\infty}\lambda^m\Gamma \Phi_{m}^{*}\Gamma \bar z^m$$ Then 
\begin{align*}
    z^n\Phi(z)e_j
    &=z^n\Big [\sum_{m=1}^{\infty}\Phi_{m}z^m+\sum_{m=0}^{\infty}\lambda^m\Gamma \Phi_{m}^{*}\Gamma \bar z^m\Big ]e_j\\
    &=\Big [\sum_{m=1}^{\infty}\Phi_{m}z^{n+m}+\sum_{m=0}^{\infty}\lambda^m\Gamma \Phi_{m}^{*}\Gamma z^{n-m}\Big ]e_j.
\end{align*}
By \eqref{coeff}, we have 
\begin{align*}
\Phi_{-n}e_j
&=\frac{1}{2\pi}\int_{0}^{2\pi}z^{n}\Phi(z)e_j dt\\
&=\frac{1}{2\pi}\int_{0}^{2\pi}\Big [\sum_{m=1}^{\infty}\Phi_{m}z^{n+m}+\sum_{m=0}^{\infty}\lambda^m\Gamma \Phi_{m}^{*}\Gamma z^{n-m}\Big ]e_jdt\\
&=\Big[\sum_{m=1}^{\infty}\Phi_{m}\frac{1}{2\pi}\int_{0}^{2\pi}z^{n+m} dt+\sum_{m=0}^{\infty}\lambda^{m}\Gamma \Phi_{m}^{*}\Gamma\frac{1}{2\pi}\int_{0}^{2\pi}z^{n-m}dt\Big]e_j.
\end{align*}
Using the orthogonality relation on the unit circle
$$\delta_{n,m}=\frac{1}{2\pi}\int_{0}^{2\pi}z^{n-m}dt=\begin{cases}
         1, ~~n=m,\\
         0, ~~n\ne m.
    \end{cases}$$
    Hence, the above integral becomes
$$\Phi_{-n}e_j=\Big[\sum_{m=1}^{\infty}\delta_{n,-m}\Phi_{m}+\sum_{m=0}^{\infty}\lambda^{m}\delta_{n,m}\Gamma \Phi_{m}^{*}\Gamma\Big] e_j.$$
If $n\geq 0$, then we have $\delta_{n,-m}=0$ and only $m=n$ term remains in the last sum. Therefore, 
$$\Phi_{-n}e_j=\lambda^{n}\Gamma\Phi_{n}^{*}\Gamma e_j. $$ It follows that $$\Phi_{-n}=\lambda^n\Gamma \Phi_{n}^{*}\Gamma 
\quad \text{for}~~n\geq 0.$$ Therefore, by (ii), $T_{\Phi}$ is complex symmetric. 

\medskip

\ni $(i)\Leftrightarrow (iv)$ Recall that
\begin{align}
\Phi(z)
&=\sum_{n=1}^{\infty}  \Phi_{n}z^n+\sum_{n=0}^{\infty}\Phi_{-n}\bar z^n\nn\\
&=\sum_{n=1}^{\infty}  \Phi_{n}z^{n}+\sum_{n=0}^{\infty}\Phi_{-n}\bar z^{n}\nn\\
&=\Phi_{+}(z)+\Phi_{-}(z)^{*},\label{decomp}
\end{align}
where
$$\Phi_{+}(z)=\sum_{n=1}^{\infty}  \Phi_{n}z^{n}\quad \text{and} \quad \Phi_{-}(z)=\sum_{n=0}^{\infty}\Phi_{-n}^{*}z^n.$$
Assume that $T_{\Phi}$ is complex symmetric with respect to $C_{\mu, \lambda}$. Then we have 
$$\Phi_{-n}=\lambda^{n}\Gamma \Phi_{n}^{*}\Gamma \quad \text{for}~~n\geq 0.$$ 
Therefore,
\begin{align*}
\Phi_{-}(z)^{*}
&=\sum_{n=0}^{\infty}\Phi_{-n}\bar z^{n}\\
&=\Gamma\Phi_{0}^{*}\Gamma+\sum_{n=1}^{\infty}\lambda^n\Gamma \Phi_{n}^{*}\Gamma \bar z^n\\
&=\Phi_{0}+\sum_{n=1}^{\infty}\Gamma \Phi_{n}^{*}\bar \lambda^{n} z^n\Gamma \\
&=\Phi_{0}+\Gamma \Big[\sum_{n=1}^{\infty} \Phi_{n}^{*}\bar \lambda^n z^n\Big]\Gamma \\
&=\Phi_{0}+\Gamma \Big[\sum_{n=1}^{\infty} \Phi_{n} \lambda^n \bar z^n\Big]^{*}\Gamma \\
&=\Phi_{0}+\Gamma \Phi_{+}(\lambda \bar z)^{*}\Gamma.
\end{align*}
By \eqref{decomp}, it follows that 
$$\Phi(z)=\Phi_{+}(z)+\Phi_{0}+\Gamma \Phi_{+}(\lambda \bar z)^{*}\Gamma,$$
where $\Phi_{+}\in zH^2(\oL)$.

The converse is immediate and is left to the reader. 
\end{proof}

\medskip

Theorem \ref{Gen2} generalizes Theorem 2.11 of \cite{KKL} by replacing the conjugation on $H^2(\mathbb{C}^2)$ with the conjugation $C_{\mu, \lambda}$, which is generated by a conjugation $\Gamma$ on $E$. Here, we present a more general proof of that theorem which does not involve scalar Toeplitz operators (Toeplitz operators on the classical Hardy-Hilbert space $H^2$).

\begin{theorem}\label{Gen2}
Let $\Phi\in L^{\infty}(\oL)$. Then $T_{\Phi}$ is complex symmetric with the conjugation $C_{\mu, \lambda}$ if and only if $M_{\Phi}C_{\mu, \lambda}=C_{\mu, \lambda}M_{\Phi^{*}}$.
\end{theorem}
\begin{proof}
Let $T_{\Phi}$ is complex symmetric with the conjugation $C_{\mu, \lambda}$. Then, by Theorem \ref{CSTO}, we have 
$$\Gamma\Phi(\lambda \bar z)\Gamma=\Phi(z)^{*}.$$ 
Then, for $k=0,\pm1, \pm2,...,$ we have
\begin{align*}
M_{\Phi}C_{\mu, \lambda}(z^k e_j)
&=\mu M_{\Phi}( \Gamma (\lambda ^k \bar z^k e_j))\\
&=\mu \Phi(z) \Gamma (\lambda ^k \bar z^k e_j)\\
&=\mu \Gamma \Phi(\lambda \bar z)^{*}\Gamma^{2} (\lambda ^k \bar z^k e_j)\\
&=\mu \Gamma \Phi(\lambda \bar z)^{*} (\lambda ^k \bar z^k e_j)\\
&=C_{\mu, \lambda}(\Phi(z)^{*}z^k e_j)\\
&=C_{\mu, \lambda}M_{\Phi}^{*}(z^k e_j).
\end{align*}
It follows that $$M_{\Phi}C_{\mu, \lambda}f=C_{\mu, \lambda}M_{\Phi^*}f\quad \text{for all}\quad f\in L^2(E).$$
Conversely, suppose that $M_{\Phi}C_{\mu, \lambda}=C_{\mu, \lambda}M_{\Phi^*}$. Then, for $f\in L^2(E)$, we have
\begin{align*}
T_{\Phi}C_{\mu, \lambda}f
&=P_{+}(M_{\Phi}C_{\mu, \lambda}f)\\
&=P_{+}(C_{\mu, \lambda}M_{\Phi^*}f)\\
&=C_{\mu, \lambda}P_{+}M_{\Phi^*}f\\
&=C_{\mu, \lambda}T_{\Phi^*}f.
\end{align*}
This completes the proof. 
\end{proof}

\section{Truncated Toeplitz operators}\label{s4}
In this section we discuss the applications of the conjugation $C_{\mu, \lambda}$ to matrix valued truncated Toeplitz operators.

 In the scalar case, each model space $K_{\theta}$ is equipped with a natural conjugation $C_{\theta}$, defined in terms of boundary functions by
$(C_\theta f)(z)=\theta(z)\overline{z}\overline{f(z)}$. The conjugation $C_\theta$ was considered in \cite{sar}. If $\Theta\in H^\infty(\oL)$ is an inner function and $\Gamma$ is a conjugation on $\hsp$, we can similarly define $C_{\Theta}:L^2(\hsp)\to L^2(\hsp)$ by
\begin{equation}\label{CTheta}
(C_{\Theta}f)(z)=\Theta(z)\overline{z} \Gamma(f(z))\quad \text{a.e. on }~~\mathbb{T}.
\end{equation}
Although $C_{\Theta}$ is obviously an antilinear isometry, it is in general not an involution. A simple computation shows that $C_{\Theta}$ is an involution (and therefore a conjugation) if and only if $\Theta(z)\Gamma\Theta(z)\Gamma=I_{\hsp}$ a.e. on $\mathbb{T}$, i.e., if and only if $\Theta$ is $\Gamma$-symmetric.

If $\Theta$ is $\Gamma$-symmetric, then $C_{\Theta}(\Theta H^2(E))=H^2(E)^{\perp}$ and so
$$C_{\Theta}(K_\Theta)=K_\Theta.$$

\begin{proposition}\label{Mphi}
   Let $\Phi\in L^{\infty}(\oL)$, and let  $\Theta$ be $\Gamma$-symmetric inner function. Let $C_{\Theta}$ be the corresponding conjugation on $L^{2}(E)$ such that $\Phi\Theta=\Theta\Phi$ and $\Phi$ is also $\Gamma$-symmetric. Then we have $C_{\Theta}M_{\Phi}C_{\Theta}=M_{\Phi}^{*}$.
\end{proposition}
\begin{proof}
 Here we consider $C_{\Theta}$ as a conjugation on $L^2(E)$. Let $f\in L^2(E)$; we have 
 \begin{align*}
C_{\Theta}M_{\Phi}C_{\Theta}f
&=C_{\Theta}M_{\Phi}(\Theta(z)\bar z \Gamma (f(z))\\
&=C_{\Theta}(\Phi(z)\Theta(z)\bar z \Gamma (f(z)))\\
&=\Theta(z)\bar z\Gamma(\Theta(z)\Phi(z)\bar z \Gamma (f(z)))\\
&=\Theta(z)\Gamma(\Theta(z)\Phi(z) \Gamma (f(z)))\\
&=\Theta(z)\Theta(z)^{*}\Gamma(\Phi(z) \Gamma (f(z)))\\
&=\Gamma\Phi(z) \Gamma (f(z))\\
&=\Phi(z)^{*}(f(z))\\
&=M_{\Phi}^{*}f
 \end{align*}
\end{proof}

The action of $C_{\mu, \lambda}$ on $K_{\Theta}$ is given by the following lemma. 

\begin{lemma}\label{inv}
$C_{\mu, \lambda}K_{\Theta}=K_{\Theta_{\lambda}}$, where $\Theta_{\lambda}(z)=\Gamma \Theta(\lambda \bar z)\Gamma$. 
\end{lemma}
\begin{proof}
Let $\Gamma$ be a conjugation on $E$ and $\mu, 
\lambda\in \mathbb{T}$. Then we first have to show that $C_{\mu, \lambda}(\Theta H^2(E))\subset \Theta_{\lambda}H^2(E)$. Let $f\in \Theta H^2(E)$, then $f=\Theta h$ for some $h\in H^2(E)$. Furthermore, 
\begin{align*}
C_{\mu, \lambda}f(z)
&=C_{\mu, \lambda}(\Theta(z) h(z))\\
&=\mu \Gamma\Theta(\lambda \bar z) h(\lambda \bar z)\\
&=\mu \Gamma\Theta(\lambda \bar z)\Gamma \Gamma h(\lambda \bar z)\\
&=\Theta_{\lambda}(z)\mu \Gamma   h(\lambda \bar z)\\
&=\Theta_{\lambda} C_{\mu, \lambda}h(z)\in \Theta_{\lambda}H^2(E).
\end{align*}
Hence, we have $C_{\mu, \lambda}(\Theta H^2(E))\subset \Theta_{\lambda} H^2(E)$. Similarly, we have $C_{\mu, \lambda}(\Theta_{\lambda}H^2(E))\subset \Theta H^2(E)$. Hence, we have 
$$C_{\mu, \lambda} (\Theta H^2(E))=\Theta_{\lambda }H^2(E).$$ Since, $C_{\mu, \lambda}$ being a conjugation, is an isometry on $H^2(E)$, we have $$C_{\mu, \lambda} K_{\Theta}=K_{\Theta_{\lambda }}.$$
\end{proof}
\begin{corollary}
$C_{\mu, \lambda}K_{\Theta}=K_{\Theta}$ if $\Gamma \Theta(\lambda \bar z )\Gamma=\Theta(z)$.
\end{corollary}
\begin{proof}
The proof follows immediately by Lemma \ref{inv}.
\end{proof}

\begin{lemma}
  Let $\Theta$ be a $\Gamma$-symmetric inner function. Then the operator $\widetilde{U}=C_{\mu, \lambda}C_{\Theta}: L^2(E)\to L^2(E)$ is a unitary operator. 
\end{lemma}
\begin{proof}
The proof follows from \cite[Lemma 1]{GPII}.
\end{proof}

 If $P_{\Theta}$ is orthogonal projection on $K_{\Theta}$ and $P_{\Theta_{\lambda}}$ is orthogonal projection on $K_{\Theta_{\lambda}}$, then by the intertwining property of $\widetilde{U}=C_{\mu, \lambda}C_{\Theta}$, we have 
$\widetilde{U}P_{\Theta}=P_{\Theta_{\lambda}}\widetilde{U}.$ The intertwining property of $\widetilde{U}$ shows that $\widetilde{U}$ maps the model space $K_{\Theta}$ onto $K_{\Theta_{\lambda}}$. Let $U$ denote the restriction of $\widetilde{U}$ from $K_{\Theta}$ to $K_{\Theta_{\lambda}}$. Since 
$U$ is also unitary, it allows one to conjugate a matrix valued truncated Toeplitz operator on $K_{\Theta}$ into a matrix valued truncated Toeplitz operator on $K_{\Theta_{\lambda}}$. This intertwining property is used in the following theorem.

%For $f\in K_{\Theta}$ the action of $U$ is given by the following formula
%\begin{align*}
%Uf
%&=C_{\mu, \lambda}C_{\Theta}f\\
%&=C_{\mu, \lambda}(\Theta(z)\bar z \Gamma f(z))\\
%&=\mu \Gamma (\Theta(\lambda \bar z)\overline {\lambda \bar z} \Gamma f(\lambda \bar z))\\
%&=\mu \bar \lambda z \Gamma (\Theta(\lambda \bar z) \Gamma f(\lambda \bar z)).
%\end{align*}
%For $g\in K_{\Theta_{\lambda}}$ the action of $U^{*}$ is given by
%\begin{align*}
   % U^{*}g
   % &=C_{\Theta}C_{\mu, \lambda}g\\
   % &=\bar \mu \bar z \Theta(z)g(\lambda \bar z). 
%\end{align*}

\begin{theorem}
  Let $\Phi\in L^{\infty}(\oL)$ and $\Theta$ be an inner function such that $\Phi$ and $\Theta$ are $\Gamma$-symmetric and $\Phi\Theta=\Theta\Phi$. Then, for  $A_{\Phi}^{\Theta}\in\mathcal{MT}(\Theta) $ and  $\widetilde{U}|_{K_{\Theta}}=U$, we have 
  $$UA_{\Phi}^{\Theta}U^{*}=A_{\Psi}^{\Theta_{\lambda}},$$ 
  where $\Psi(z)=\Phi(\lambda \bar z)$. 
\end{theorem}
\begin{proof}
    Let $\Phi$ and $\Theta$ be $\Gamma$-symmetric  and suppose that $\Phi\Theta=\Theta\Phi$. Then, by Proposition \ref{Mphi}, we have 
    \begin{align*}
        UA_{\Phi}^{\Theta}U^{*}f
        &=UP_{\Theta}M_{\Phi}U^{*}f\\
        &=P_{\Theta_{\lambda}}UM_{\Phi} U^{*}f\\
        &=P_{\Theta_{\lambda}}C_{\mu, \lambda}C_{\Theta}M_{\Phi} C_{\Theta}C_{\mu, \lambda}f\\
        &=P_{\Theta_{\lambda}}C_{\mu, \lambda}M_{\Phi^{*}} C_{\mu, \lambda}f\\
         &=P_{\Theta_{\lambda}}C_{\mu, \lambda}M_{\Phi^{*}} (\mu \Gamma f(\lambda \bar z))\\
         &=P_{\Theta_{\lambda}}(\bar \mu C_{\mu, \lambda}(\Phi^{*}(z) \Gamma f(\lambda \bar z)))\\
         &=P_{\Theta_{\lambda}}(\bar \mu \mu\Gamma\Phi^{*}(\lambda \bar z) \Gamma f(\lambda \overline{\lambda \bar z})))\\
         &=P_{\Theta_{\lambda}}(\Gamma\Phi^{*}(\lambda \bar z) \Gamma f(z))\\
          &=P_{\Theta_{\lambda}}(\Phi(\lambda \bar z) f(z))\\
         &=A_{\Psi}^{\Theta_{\lambda}}f,
    \end{align*}
    where $\Psi(z)=\Phi(\lambda \bar z)$. 
\end{proof}

\vspace{0.5cm}

\noindent{\bf{Conflict of interest}}\\
The author declare no conflicts of interest.

\vspace{0.5cm}

\noindent{\bf{Data availability}}\\
Not applicable.

\vspace{0.5cm}

%\noindent{\bf{Acknowledgment}}\\
%I would like to thank the reviewers for their valuable comments and suggestions.

\end{document}